\documentclass[final,3p,times]{elsarticle}
\usepackage{amssymb}
\usepackage{amsmath}
\usepackage[caption=false,font=footnotesize]{subfig}
\usepackage{amsfonts}
\usepackage{graphicx}
 \graphicspath{{./pics/}}
 \DeclareGraphicsExtensions{.pdf,.jpg}
\usepackage{tikz}
 \usetikzlibrary{arrows,automata,shapes}
 \usetikzlibrary{positioning}
 \usetikzlibrary{intersections}
\usepackage[textwidth=1.9cm,textsize=footnotesize]{todonotes}

\newtheorem{theorem}{Theorem}
\newtheorem{lemma}[theorem]{Lemma}
\newtheorem{cor}[theorem]{Corollary}
\newtheorem{proposition}[theorem]{Proposition}
\newdefinition{definition}[theorem]{Definition}
\newtheorem{algorithm}[theorem]{Algorithm}
\newdefinition{problem}[theorem]{Problem}
\newdefinition{example}[theorem]{Example}
\newdefinition{remark}[theorem]{Remark}
\newproof{proof}{Proof}

\newcommand{\eps}{\varepsilon}
\newcommand{\supC}{\mbox{$\sup {\rm C}$}}
\newcommand{\supCC}{\mbox{$\sup {\rm cC}$}}

\newcommand{\supcn}{\mbox{$\sup {\rm CN}$}}
\newcommand{\supCN}{\mbox{$\sup {\rm CN}$}}
\newcommand{\supccn}{\mbox{$\sup {\rm cCN}$}}
\newcommand{\supCCN}{\mbox{$\sup {\rm cCN}$}}

\journal{arxiv.org}

\begin{document}

\begin{frontmatter}

\title{A Relaxed Framework for Coordination Control\\ of Discrete-Event Systems}

\author[inst]{Jan~Komenda}
 \ead{komenda@ipm.cz}
 \address[inst]{Institute of Mathematics, Czech Academy of Sciences, {\v Z}i{\v z}kova 22, 616 62 Brno, Czech Republic}
\author[inst,tud]{Tom{\' a}{\v s}~Masopust}
 \ead{masopust@math.cas.cz}
 \address[tud]{Technische Universit\"at Dresden, Germany}
\author[vSCR]{Jan H. van Schuppen}
 \ead{jan.h.van.schuppen@xs4all.nl}
 \address[vSCR]{Van Schuppen Control Research,
  Gouden Leeuw 143,
  1103 KB Amsterdam,
  The Netherlands}

\begin{abstract}
  In this paper, we simplify the coordination control approach by removing the supervisor for the coordinator from the closed-loop system and relax the restrictions placed on a coordinator. This relaxation results in the simplification of the whole coordination control framework, including the notions of conditional controllability, conditional observability, and conditional normality. Compared to our previous work, the role of the supervisor on a coordinator alphabet is postponed until the final stage of the coordination control synthesis. This completes and clarifies our previous results, while all the fundamental theorems remain valid in the relaxed framework. Unlike previous approaches we can  always compute a conditionally controllable sublanguage without any restricting conditions we have used before.
\end{abstract}
\end{frontmatter}

\section{Introduction}
  In this paper, we further investigate supervisory control of concurrent discrete-event systems (DES) with a coordinator. Discrete-event systems modeled as finite automata have been widely studied by P. J. Ramadge and W. M. Wonham \cite{RW87}. Large DES are typically formed in a compositional way as products of local subsystems (smaller finite automata) that communicate with each other in a synchronous~\cite{CL08} or asynchronous~\cite{LauriePhilippe} way. Such systems are often called modular DES.

  Supervisory control theory aims to guarantee that the control specifications consisting of safety and of nonblockingess are satisfied in the controlled (closed-loop) system. Safety means that the language of the closed-loop system is included in a prescribed specification language, and nonblockingness means that all generated strings can always be completed to a marked string. Supervisory control is realized by a supervisor that runs in parallel with the system and imposes the specification by disabling, at each state, some of the controllable events in a feedback manner. Since only controllable specification languages can be exactly achieved, one of the key issues in the supervisory control synthesis is a computation of the supremal controllable sublanguage of the specification language, from which the supervisor can be constructed.
 
  Unfortunately, the number of states of a modular DES grows exponentially with respect to the number of components, which limits the applicability of the centralized supervisory control synthesis to relatively small systems. Moreover, the safety specification is considered as a global property that is independent on the product structure of the system. The purely decentralized control (an independent construction of a supervisor for each subsystem) is only applicable for a local (decomposable) specification. Therefore, in~\cite{KvS08} and~\cite{automatica2011}, see also~\cite{JDEDS}, we have proposed a coordination control framework with local supervisors communicating with each other via a coordinator. The framework is based on new concepts of conditional decomposability and conditional controllability. Conditional controllability of the specification is proven to be an equivalent condition for the existence of local supervisors that can achieve the specification in cooperation with the coordinator. 

  In this paper, the basic framework of coordination control of modular DES from~\cite{KvS08} and~\cite{JDEDS} is further simplified. Our original framework, where the closed-loop for the coordinator part of the specification is added to the plant of local supervisors, has been motivated by the antimonotonicity of the basic supervisory control operators (supremal controllable and normal sublanguages) with respect to the plant. More precisely, if the specification language is fixed, then decreasing the plant language leads to increasing the supervisors given by supremal controllable and normal sublanguages, hence permissiveness is potentially increased. However, it turns out that in the original framework of coordination control as first introduced in~\cite{KvS08}, decreasing the plant $G_i\parallel G_k$ for local supervisors by replacing $G_k$ by, in general, a smaller closed-loop $S_k/G_k$ does not change the permissiveness of local closed-loops $S_{i+k}/[G_i \parallel (S_k/G_k)]$, because of the transitivity of controllability (note that closed-loop $S_k/G_k$ is always controllable with respect to $G_k$). More precisely, a language $M\subseteq G_i \parallel (S_k/G_k)$ is controllable with respect to $G_i \parallel (S_k/G_k)$ if and only if it is controllable with respect to  $G_i \parallel G_k$. For this reason, the supervisor for the coordinator does not help and it is replaced by a supervisor on the coordinator alphabet at the very end of the coordination control algorithm. Conditional controllability, conditional observability, and conditional normality are then correspondingly relaxed by dropping the requirements on controllability, observability, and normality on the coordinator alphabet, respectively. 
 
  Our latest results show that although controllability of the resulting closed-loop on the coordinator alphabet is a key condition for the ability to compute the supremal conditionally controllable sublanguage, this controllability can actually always be guaranteed by constructing a supervisor that restricts the resulting coordination control closed-loop so that it becomes controllable on the coordinator alphabet. It should be noted that without this final restriction the computed language would be larger and only controllable, but not conditional controllable in general. Conditional controllability is, however, a necessary and sufficient condition on a language to be achievable in the coordination control architecture using coordinator and local supervisors. This is true both in our previous architecture with the original definition of conditional controllability and in the simplified architecture of this paper and the relaxed definition of conditional controllability. 
 
  Due to these simplifications the complexity of coordination control synthesis is reduced, and we can always compute a conditionally controllable sublanguage of the specification in a distributed way.

  The paper is organized as follows. In section~\ref{sec:preliminaries}, the necessary notation is introduced and the basic elements of supervisory control theory are recalled. The new relaxed framework of coordination control is presented in Section~\ref{sec:controlsynthesis}, where the simplified notion of conditional controllability (Definition~\ref{def:conditionalcontrollability}) and conditional observability (Definition~\ref{def:conditionalobservability}) are proposed. Section~\ref{sec:procedure} is devoted to the main results of the paper, namely to the distributed computation of supremal conditionally controllable and conditionally normal sublanguages. It discusses the computational benefit of the new coordination control framework, and presents an illustrative example. Finally, concluding remarks are given in Section~\ref{sec:conclusion}.

\section{Preliminaries and definitions}\label{sec:preliminaries}
  Let $\Sigma$ be a finite nonempty set of events (called an {\em alphabet}), and let $\Sigma^*$ denote the set of all finite words over the alphabet $\Sigma$. The empty word is denoted by $\eps$. A {\em language\/} over an alphabet $\Sigma$ is a subset of $\Sigma^*$. The prefix closure of a language $L$ over the alphabet $\Sigma$ is the set of all prefixes of all words from $L$, that is, $\overline{L}=\{w\in \Sigma^* \mid \text{there exists } u \in\Sigma^* \text{ such that } wu\in L\}$. A language $L$ is {\em prefix-closed\/} if $L=\overline{L}$. For more details, the reader is referred to~\cite{CL08,Won04}. Discrete-event systems (DES) are formed as synchronous products of generators. 
 
  \paragraph{Generators}
  A {\em generator\/} of a DES is the structure $G=(Q,\Sigma, f, q_0, Q_m)$, where $Q$ is the finite set of states, $\Sigma$ is the finite nonempty set of events (an alphabet of the generator $G$), $f: Q \times \Sigma \to Q$ is the partial transition function, $q_0 \in Q$ is the initial state, and $Q_m\subseteq Q$ is the set of marked states. The transition function $f$ can be extended to the domain $Q \times \Sigma^*$ in the usual way by induction. The behavior of the generator $G$ is described in terms of languages. The language {\em generated\/} by $G$ is the set $L(G) = \{s\in \Sigma^* \mid f(q_0,s)\in Q\}$, and the language {\em marked\/} by $G$ is the set $L_m(G) = \{s\in \Sigma^* \mid f(q_0,s)\in Q_m\}$. Obviously, $L_m(G)\subseteq L(G)$.

  \paragraph{Natural projections}
  For $\Sigma_0\subseteq \Sigma$, a {\em (natural) projection\/} is a mapping $P: \Sigma^* \to \Sigma_0^*$, which deletes from any word all letters that belong to $\Sigma\setminus \Sigma_0$. Formally, it is a homomorphism defined by $P(a)=\eps$, for $a$ in $\Sigma\setminus \Sigma_0$, and $P(a)=a$, for $a$ in $\Sigma_0$. It is then extended (as a homomorphism for concatenation) from letters to words by induction. The {\em inverse image\/} of $P$ is denoted by $P^{-1}:\Sigma_0^* \to 2^{\Sigma^*}$. For three event sets $\Sigma_i$, $\Sigma_j$, $\Sigma_\ell$, subsets of $\Sigma$, the notation $P^{i+j}_{\ell}$ is used to denote the projection from $(\Sigma_i\cup \Sigma_j)^*$ to $\Sigma_\ell^*$. If $\Sigma_i\cup \Sigma_j=\Sigma$, we simplify the notation to $P_\ell$. Similarly, the notation $P_{i+k}$ stands for the projection from $\Sigma^*$ to $(\Sigma_i\cup\Sigma_k)^*$. The projection of a generator $G$, denoted by $P(G)$, is a generator whose behavior satisfies $L(P(G))=P(L(G))$ and $L_m(P(G))=P(L_m(G))$. It is defined using the standard subset construction, cf.~\cite{CL08}.
 
  \paragraph{Synchronous product and nonconflictingness}
  The synchronous product of languages $L_1$ over $\Sigma_1$ and $L_2$ over $\Sigma_2$ is defined as the language $L_1\parallel L_2=P_1^{-1}(L_1) \cap P_2^{-1}(L_2)$, where $P_i: (\Sigma_1\cup \Sigma_2)^*\to \Sigma_i^*$ is a projection, for $i=1,2$. A similar definition for generators can be found in~\cite{CL08}. For generators $G_1$ and $G_2$, it holds that $L(G_1 \parallel G_2) = L(G_1) \parallel L(G_2)$ and $L_m(G_1 \parallel G_2)= L_m(G_1) \parallel L_m(G_2)$. 
  
  Languages $K$ and $L$ are {\em synchronously nonconflicting\/} if $\overline{K \parallel L} = \overline{K} \parallel \overline{L}$. 
  
  \paragraph{Basic supervisory control theory}
  Now we recall the basic elements of supervisory control theory. A {\em controlled generator\/} over an alphabet $\Sigma$ is a structure $(G,\Sigma_c,P,\Gamma)$, where $G$ is a generator over the alphabet $\Sigma$, $\Sigma_c\subseteq\Sigma$ is a set of {\em controllable events}, $\Sigma_{u} = \Sigma \setminus \Sigma_c$ is the set of {\em uncontrollable events}, $P:\Sigma^*\to \Sigma_o^*$ is the natural projection from $\Sigma$ to the set of {\em observable events\/} $\Sigma_o$, and $\Gamma = \{\gamma \subseteq \Sigma \mid \Sigma_{u} \subseteq \gamma\}$ is the {\em set of control patterns}. A {\em supervisor\/} for the controlled generator $(G,\Sigma_c,P,\Gamma)$ is a map $S:P(L(G)) \to \Gamma$. The {\em closed-loop system\/} associated with the controlled generator $(G,\Sigma_c,P,\Gamma)$ and the supervisor $S$ is defined as the smallest language $L(S/G)$ such that $\eps \in L(S/G)$, and if $s \in L(S/G)$, $sa\in L(G)$, and $a \in S(P(s))$, then $sa \in L(S/G)$. The marked language of the closed-loop system is defined as the set of marked strings of the uncontrolled generator that survive under supervision, that is, $L_m(S/G) = L(S/G)\cap L_m(G)$. The intuition is that the supervisor disables some of the controllable transitions of the generator $G$, but it never disables an uncontrollable transition. If $\overline{L_m(S/G)}=L(S/G)$, then the supervisor $S$ is called {\em nonblocking}. In the automata framework, where a supervisor $S$ has a finite representation as a generator, the closed-loop system is a synchronous product of the supervisor and the plant. Thus, we can write the closed-loop system as $L(S/G)=L(S) \parallel L(G)$.

  \paragraph{Control objectives of supervisory control}
  Control objectives of supervisory control are defined using a specification language $K$, and the goal of supervisory control is to find a nonblocking supervisor $S$ such that $L_m(S/G)=K$. In the monolithic case, that is, if the plant $G$ is a single generator, such a supervisor exists if and only if $K$ is {\em controllable\/} with respect to $L(G)$ and $\Sigma_u$ (that is, $\overline{K}\Sigma_u\cap L\subseteq \overline{K}$), {\em $L_m(G)$-closed} (that is, $K = \overline{K}\cap L_m(G)$), and {\em observable\/} with respect to $L(G)$, $\Sigma_o$, and $\Sigma_c$ (that is, for all $s\in \overline{K}$ and for all $\sigma \in \Sigma_c$, $(s\sigma \notin \overline{K})$ and $(s\sigma \in L(G))$ implies that $P^{-1}[P(s)]\sigma \cap \overline{K} = \emptyset$), cf.~\cite{CL08}. If $\Sigma_o=\Sigma$, we say that the system is with {\em full observation}; otherwise, it is with {\em partial observation}.

  If the specification fails to satisfy controllability or observability, controllable and observable sublanguages of the specification are considered instead. However, since, in general, supremal observable sublanguages do not exist, {\em normality\/} is used instead of observability. A language $K\subseteq L(G)$ is {\em normal\/} with respect to plant $G$ and the partial observation $P$ if $\overline{K} = P^{-1}[P(\overline{K})]\cap L(G)$. It is known that normality implies observability~\cite{CL08}. Thus, for specifications that are either not controllable or not observable, supremal controllable and normal sublanguages are considered instead.\footnote{In~\cite{caiCDC13}, relative observability has been introduced that can be used instead of normality. Relative observability is a condition stronger than observability and weaker than normality. It was shown to be closed under language unions, hence supremal relatively observable sublanguages exist, and relative observability can thus replace normality in practical applications. Let $K \subseteq C \subseteq L(G)$. The language $K$ is relatively observable or {\em $C$-observable\/} with respect to a plant $G$ and a projection $P:\Sigma^*\to\Sigma_o^*$ if for all words $s, s'\in \Sigma^*$ such that $P(s) = P(s')$, it holds that, for all $\sigma\in \Sigma$, if $s\sigma \in \overline{K}$, $s' \in \overline{C}$ and $s'\sigma \in L(G)$, then $s'\sigma \in \overline{K}$. Note that for $C=K$ the definition coincides with the definition of observability. An interested reader can find more about relative observability in coordination control in~\cite{KomendaMS14a}.}
  Let $\supC(K,L(G),\Sigma_u)$ denote the supremal controllable sublanguage of $K$ with respect to $L(G)$ and $\Sigma_u$, which always exists and is equal to the union of all controllable sublanguages of $K$, cf.~\cite{Won04}. Similarly, for the partial observation case, let $\supcn(K,L(G),\Sigma_u,P)$ denote the supremal controllable and normal sublanguage of $K$ with respect to $L(G)$, $\Sigma_u$, and $P$. The supremal controllable and normal sublanguage always exists and equals the union of all controllable and normal sublanguages of $K$, see, e.g,~\cite{CL08}. A formula for computing supremal controllable and normal sublanguages can be found in~\cite{brandt}.

  \paragraph{Observer and OCC/LCC properties}
  The projection $P:\Sigma^* \to \Sigma_0^*$, with $\Sigma_0\subseteq \Sigma$, is an {\em $L$-observer\/} for $L\subseteq \Sigma^*$ if, for all $t\in P(L)$ and $s\in \overline{L}$, $P(s)$ is a prefix of $t$ implies that there exists $u\in \Sigma^*$ such that $su\in L$ and $P(su)=t$. 
 
  The projection $P:\Sigma^* \to \Sigma_0^*$ is {\em output control consistent} (OCC) for $L\subseteq\Sigma^*$ if for every $s\in \overline{L}$ of the form $s=\sigma_1\sigma_2\dots\sigma_\ell$ or $s = s'\sigma_0 \sigma_1 \dots \sigma_\ell$, for $\ell\ge 1$, where $s'\in \Sigma^*$, $\sigma_0, \sigma_\ell\in \Sigma_k$, and $\sigma_i \in \Sigma\setminus \Sigma_k$, for $i=1,2,\dots,\ell-1$, if $\sigma_\ell \in \Sigma_{u}$, then $\sigma_i \in \Sigma_{u}$, for all $i=1,2,\dots,\ell-1$. 
  The OCC condition can be replaced by a weaker condition called local control consistency (LCC) discussed in~\cite{SB11,SB08}, see~\cite{JDEDS}. Let $L$ be a prefix-closed language over the alphabet $\Sigma$, and let $\Sigma_0$ be a subset of $\Sigma$. The projection $P:\Sigma^*\to \Sigma_0^*$ is {\em locally control consistent\/} (LCC) with respect to a word $s\in L$ if for all events $\sigma_u\in \Sigma_0\cap \Sigma_u$ such that $P(s)\sigma_u\in P(L)$, it holds that either there does not exist any word $u\in (\Sigma\setminus \Sigma_0)^*$ such that $su\sigma_u \in L$, or there exists a word $u\in (\Sigma_u\setminus \Sigma_0)^*$ such that $su\sigma_u \in L$. The projection $P$ is LCC with respect to $L$ if $P$ is LCC for all words of $L$.

  \paragraph{Conditional decomposability}
  A language $K$ is {\em conditionally decomposable\/} with respect to event sets $\Sigma_1$, $\Sigma_2$, and $\Sigma_k$, where $\Sigma_1\cap\Sigma_2\subseteq\Sigma_k\subseteq \Sigma_1\cup\Sigma_2$, if $K = P_{1+k}(K)\parallel P_{2+k}(K)$, where $P_{i+k}:(\Sigma_1\cup\Sigma_2)^*\to (\Sigma_i\cup\Sigma_k)^*$ is a projection, for $i=1,2$.
  There always exists an extension of $\Sigma_k$ which satisfies this condition; $\Sigma_k = \Sigma_1 \cup \Sigma_2$ is a trivial example. There exists a polynomial algorithm to check this condition, and to extend the event set to satisfy the condition, see~\cite{scl12}. To find the minimal extension with respect to set inclusion is NP-hard~\cite{JDEDS}.

\section{Coordination control synthesis}\label{sec:controlsynthesis}
  In this section, we first reformulate the basic coordination control problem. Unlike the original approach~\cite{automatica2011,JDEDS,scl11}, the supervisor for a coordinator is not included in the coordinated closed-loop. Consequently, the closed-loop for the coordinator is replaced by the coordinator itself in the plants for local supervisors. A supervisor on the coordinator alphabet, however, plays a significant role at the end of the coordination control synthesis.
 
  \begin{problem}[Relaxed coordination control problem]\label{problem:controlsynthesis}
    Consider generators $G_1$ and $G_2$ over the alphabets $\Sigma_1$ and $\Sigma_2$, respectively. Let $\Sigma_k$ be an alphabet such that $\Sigma_1\cap \Sigma_2 \subseteq \Sigma_k \subseteq \Sigma_1\cup\Sigma_2$. A generator $G_k$ over the alphabet $\Sigma_k$ is called a {\em coordinator}. Assume that a specification $K \subseteq L_m(G_1 \parallel G_2 \parallel G_k)$ and its prefix-closure $\overline{K}$ are conditionally decomposable with respect to $\Sigma_1$, $\Sigma_2$, and $\Sigma_k$. The aim of coordination control is to determine nonblocking supervisors $S_1$ and $S_2$ such that $L_m(S_i/ [G_i \parallel G_k ])\subseteq P_{i+k}(K)$, for $i=1,2$, and such that the closed-loop system with the coordinator satisfies $L_m(S_1/ [G_1 \parallel G_k ]) ~ \parallel ~ L_m(S_2/ [G_2 \parallel G_k ]) = K$.
  \end{problem}

  Compared to our previous work~\cite{automatica2011,JDEDS}, the construction of a coordinator that does not affect the behavior of the plant remains unchanged~\cite{JDEDS}.

  \begin{algorithm}[Construction of a coordinator]\label{algorithm}
    Consider generators $G_1$ and $G_2$ over the alphabets $\Sigma_1$ and $\Sigma_2$, respectively, and let $K$ be a specification. We compute the event set $\Sigma_k$ and the coordinator $G_k$ as follows.
    \begin{enumerate}
      \item Let $\Sigma_k = \Sigma_1\cap \Sigma_2$ be the set of all shared events of the generators $G_1$ and $G_2$.
      \item Extend the alphabet $\Sigma_k$ so that $K$ and $\overline{K}$ are conditional decomposable with respect to $\Sigma_1$, $\Sigma_2$, and $\Sigma_k$, cf.~\cite{scl12} for a polynomial algorithm.
      \item Define the coordinator $G_k$ as $G_k = P_k(G_1) \parallel P_k(G_2)$, where $P_k:(\Sigma_1\cup\Sigma_2)^*\to \Sigma_k^*$.
    \end{enumerate}
  \end{algorithm}
 
  It is the folklore that, in general, the computation of a projected generator can be exponential. However, it is also known that if the projection satisfies the observer property, then the projected generator is smaller than the original generator. Therefore, we might need to add a middle step to extend the event set $\Sigma_k$ so that the projection $P_k:\Sigma^* \to \Sigma_k^*$ is an $L(G_i)$-observer, for $i=1,2$, before computing the coordinator $G_k$ in step (3), cf.~\cite{pcl08,pcl12}.

  In our coordination control framework, the full observation case, conditional controllability plays the role of a necessary and sufficient condition for the existence of local supervisors that in cooperation with the coordinator achieve the specification language. Since in our relaxed framework we do not use the supervisor for a coordinator, controllability of the coordinator part $P_k(K)$ of the specifications is skipped from the relaxed definitions below, compared to our previous work.
 
  \begin{definition}[Relaxed conditional controllability]\label{def:conditionalcontrollability}
    Let $G_1$ and $G_2$ be generators over the alphabets $\Sigma_1$ and $\Sigma_2$, respectively, and let $G_k$ be a coordinator over the alphabet $\Sigma_k$. A language $K\subseteq L_m(G_1\parallel G_2\parallel G_k)$ is {\em conditionally controllable\/} for generators $G_1$, $G_2$, $G_k$ if $P_{i+k}(K)$ is controllable with respect to $L(G_i \parallel G_k)$ and $\Sigma_{i+k,u}$, where $\Sigma_{i+k,u}=(\Sigma_i\cup \Sigma_k)\cap \Sigma_u$, for $i=1,2$.
  \end{definition}
 
  \begin{definition}[Relaxed conditional observability]\label{def:conditionalobservability}
    Let $G_1$ and $G_2$ be generators over the alphabet $\Sigma_1$ and $\Sigma_2$, respectively, and let $G_k$ be a coordinator over the alphabet $\Sigma_k$. A language $K\subseteq L_m(G_1\parallel G_2\parallel G_k)$ is {\em conditionally observable\/} for generators $G_1$, $G_2$, $G_k$ and projections $Q_{i+k}: \Sigma_{i+k}^*\to \Sigma_{i+k,o}^*$, for $i=1,2$, if $P_{i+k}(K)$ is observable with respect to $L(G_i\parallel G_k)$, $\Sigma_{i+k,c}$ and $Q_{i+k}$, where $\Sigma_{i+k,o}=\Sigma_o \cap (\Sigma_i \cup \Sigma_k)$ and $\Sigma_{i+k,c}=\Sigma_c \cap (\Sigma_i \cup \Sigma_k)$.
  \end{definition}

  For the relaxed definition of conditional controllability it holds that every conditionally controllable and conditionally decomposable language is controllable~\cite{JDEDS}. We now extend this result for conditional observability.
 
  \begin{proposition}\label{prop3}
    Let $G_i$ be a generator over the alphabet $\Sigma_i$, for $i=1,2,k$, and let $G=G_1\parallel G_2\parallel G_k$. Let $K\subseteq L_m(G)$ be a specification such that $\overline{K}$ is conditionally decomposable with respect to $\Sigma_1$, $\Sigma_2$, $\Sigma_k$. If $K$ is conditionally controllable and conditionally observable for generators $G_1$, $G_2$, $G_k$ and projections $Q_{i+k}: \Sigma_{i+k}^*\to \Sigma_{i+k,o}^*$, for $i=1,2$, then the specification $K$ is controllable and observable with respect to $L(G)$, $\Sigma_u=\Sigma_{1,u}\cup\Sigma_{2,u}$, and $Q:(\Sigma_1\cup \Sigma_2)^*\to (\Sigma_{1,o}\cup\Sigma_{2,o})^*$.
  \end{proposition}
  \begin{proof}
    By assumption, the languages $\overline{P_{i+k}(K)}$ are controllable and observable with respect to $L(G_1\parallel G_k)$ and $\Sigma_{i+k,u}$, $i=1,2$. Thus, by Lemmata~\ref{feng} and~\ref{obsComposition2}, the language $\overline{K} = \overline{P_{1+k}(K)} \parallel \overline{P_{2+k}(K)}$ is controllable and observable with respect to $L(G_1 \parallel G_2 \parallel G_k)$, $\Sigma_u$ and $Q$, hence also $K$ is.
  \qed\end{proof}

  In the case of non-prefix-closed specifications, conditional closedness has been introduced. Again, compared to our previous work, the closedness on the coordinator part of the specifications is not needed in our relaxed framework.

  \begin{definition}[Relaxed conditional closedness]\label{def:conditionalclosed}
    A nonempty language $K$ over the alphabet $\Sigma$ is {\em conditionally closed\/} for generators $G_1$ and $G_2$, and a coordinator $G_k$ if $P_{i+k}(K)$ is $L_m(G_i \parallel G_k)$-closed, for $i=1,2$.
  \end{definition}

  The main theorem of coordination control with partial observation with the simplified closed-loop system and the relaxed definition of conditional observability is as follows. It has a simplified form compared with~\cite{scl11}.

  \begin{theorem}\label{thm1}\label{th:controlsynthesissafety}
    Consider the setting of Problem~\ref{problem:controlsynthesis}. There exist nonblocking supervisors $S_1$ and $S_2$ such that the closed-loop system satisfies $L_m(S_1/[G_1 \parallel G_k]) \parallel L_m(S_2/[G_2 \parallel G_k]) = K$ if and only if $K$ is
    (i) conditionally controllable for generators $G_1$, $G_2$, $G_k$, 
    (ii) conditionally closed for generators $G_1$, $G_2$, $G_k$, and 
    (iii) conditionally observable for $G_1$, $G_2$, $G_k$ and projections $Q_{1+k}$, $Q_{2+k}$ from $\Sigma_{i+k}^*$ to $\Sigma_{i+k,o}^*$, for $i=1,2$.
  \end{theorem}
  \begin{proof}
    (If) Assume that the specification $K$ satisfies the assumptions and let $G=G_1 \parallel G_2 \parallel G_k$. Since $P_{1+k}(K) \subseteq L_m(G_1\parallel G_k)$, the assumption that $K$ is conditionally controllable, conditionally closed, and conditionally observable implies that there exists a nonblocking supervisor $S_{1}$ such that $L_m(S_1/ [ G_1 \parallel G_k]) = P_{1+k}(K)$. A similar argument for $P_{2+k}(K)$ implies that there exists a nonblocking supervisor $S_2$ such that $L_m(S_2/ [G_2 \parallel G_k]) = P_{2+k}(K)$. Since $K$ and $\overline{K}$ are conditionally decomposable, it follows that $L_m(S_1/[G_1 \parallel G_k]) \parallel L_m(S_2/[G_2 \parallel G_k]) = P_{1+k}(K) \parallel P_{2+k}(K) = K$.
    
    (Only if) To prove this implication, we can write $K = L_m(S_1 \parallel G_1 \parallel S_2 \parallel G_2 \parallel G_k)$. Since $\Sigma_{1+k}\cap\Sigma_{2+k}\subseteq \Sigma_k$, the application of $P_{1+k}$ to the equation and the assumptions of Problem~\ref{problem:controlsynthesis} give that $P_{1+k}(K) \subseteq L_m(S_1/ [G_1 \parallel G_k]) \subseteq P_{1+k}(K)$. Taking $G_1 \parallel G_k$ as a new plant, the basic supervisory control theorem gives that $P_{1+k}(K)$ is controllable with respect to $L(G_1 \parallel G_k)$ and $\Sigma_{1+k,u}$, $L_m(G_1 \parallel G_k)$-closed, and observable with respect to $L(G_1\parallel G_k)$, $\Sigma_{1+k,c}$, and $Q_{1+k}$. The case of $P_{2+k}(K)$ is analogous.
  \qed\end{proof}

  The main existential result for DES with full observation then has the same form as in~\cite{JDEDS}, but with simplified definitions and the relaxed form of the coordinated system.

  \begin{cor}\label{cor:controlsynthesissafety}
    Consider the setting of Problem~\ref{problem:controlsynthesis}. There exist nonblocking supervisors $S_1$ and $S_2$ such that $L_m(S_1/[G_1 \parallel G_k]) \parallel L_m(S_2/[G_2 \parallel G_k]) = K$ if and only if $K$ is both conditionally controllable and conditionally closed for generators $G_1$, $G_2$, $G_k$.
    \qed
  \end{cor}
  
  Finally, note that in our previous work~\cite{automatica2011} we used the observer and OCC/LCC conditions to ensure that the resulting languages form a solution. We can now provide a simple explanation why these conditions are sufficient.
  \begin{proposition}[\cite{JDEDS}]\label{prop4}
    Let $L$ be a prefix-closed language over the alphabet $\Sigma$, and let $K\subseteq L$ be controllable with respect to $L$ and $\Sigma_u$. If, for $i\in\{1,2\}$, the projection $P_{i+k}$ is an $L$-observer and LCC for the language $L$, then the language $K$ is conditionally controllable.
  \qed\end{proposition}

\section{Distributed computation of conditionally controllable and conditionally normal sublanguages}
  \label{sec:procedure}
  It is a common practice in supervisory control to compute a supremal sublanguage of the specification if the given specification does not satisfy the required properties. Thus, in our case, if the specification is either not conditionally controllable or conditional observable, that is, it cannot be achieved as the resulting behavior of the coordinated system according to Theorem~\ref{th:controlsynthesissafety}, we compute the maximal sublanguage of the specification that satisfies both these conditions. Specifically, we compute the supremal conditionally controllable sublanguage, but we cannot, in general, compute the supremal conditionally observable sublanguage, since it does not always exist. Thus, we use conditional normality instead and compute the supremal conditionally controllable and conditionally normal sublanguage. As already mentioned above, a weaker condition than normality, relative observability, has recently been discovered~\cite{caiCDC13} that can be extended to coordination control in a similar way as controllability, observability, or normality~\cite{KomendaMS14a}.
  
  In this section, the sufficient conditions for distributed computations of maximally permissive solutions are unified. Namely, in~\cite{scl11}, we have used the normality condition on the coordinator alphabet, while the observer and OCC conditions have been used instead of the weaker condition of controllability proposed in this paper. The result from~\cite{scl11} for the computation of the supremal conditionally controllable and conditionally normal sublanguage is improved accordingly.

\subsection{Supremal conditionally controllable and conditionally normal sublanguages}  
  We start with the relaxed definition of conditional normality.

  \begin{definition}[Relaxed conditional normality]\label{def:conditionalnormality}
    Let $G_1$ and $G_2$ be generators over the alphabets $\Sigma_1$ and $\Sigma_2$, respectively, and let $G_k$ be a coordinator over the alphabet $\Sigma_k$. A language $K\subseteq L_m(G_1\parallel G_2\parallel G_k)$ is {\em conditionally normal\/} for generators $G_1, G_2, G_k$ and projections $Q_{i+k}: \Sigma_{i+k}^*\to \Sigma_{i+k,o}^*$ if $P_{i+k}(K)$ is normal with respect to $L(G_i\parallel G_k)$ and $Q_{i+k}$, for $i=1,2$.
  \end{definition}

  Since relaxed conditional normality implies relaxed conditional observability, we immediately have the following result.
  \begin{theorem}\label{thm0}
    Consider the setting of Problem~\ref{problem:controlsynthesis}. If the specification $K$ is conditionally controllable for generators $G_1, G_2, G_k$, conditionally closed for generators $G_1, G_2$, $G_k$, and conditionally normal for $G_1, G_2, G_k$ and projections $Q_{i+k}: \Sigma_{i+k}^* \to \Sigma_{i+k,o}^*$, for $i=1,2$, then there exist nonblocking supervisors $S_1$ and $S_2$ such that $L_m(S_1/[G_1 \parallel G_k]) \parallel L_m(S_2/[G_2 \parallel G_k]) = K$.
  \qed\end{theorem}
  
  We now show that the supremal conditionally controllable and conditionally normal sublanguages always exist even for the simplified definitions.
  \begin{theorem}\label{existence2}
    The supremal conditionally controllable and conditionally normal sublanguage of a language $K$ exists and equals to the union of all conditionally controllable and conditionally normal sublanguages of the language $K$.
  \end{theorem}
  \begin{proof}
    We show that conditional controllability and conditional normality are preserved under language union. Let $I$ be an index set, and let $K_i$, for $i\in I$, be conditionally controllable and conditionally normal sublanguages of $K\subseteq L_m(G_1\parallel G_2\parallel G_k)$. Then $P_{1+k} \left(\cup_{i\in I} \overline{K_i}\right)\Sigma_{1+k,u} \cap L(G_1 \parallel G_k) = \cup_{i\in I} \left( P_{1+k}(\overline{K_i})\Sigma_{1+k,u} \cap L(G_1 \parallel G_k) \right) \subseteq P_{1+k}\left(\cup_{i\in I} \overline{K_i}\right)$. Similarly, $P_{1+k}(\bigcup_{i\in I} K_i)$ is normal with respect to $L(G_1\parallel G_k)$ and $Q_{1+k}$ because $Q_{1+k}^{-1}Q_{1+k}P_{1+k}(\overline{\bigcup_{i\in I} K_i}) \cap L(G_1\parallel G_k) = \bigcup_{i\in I} (Q_{1+k}^{-1}Q_{1+k}P_{1+k}(\overline{K_i}) \cap L(G_1\parallel G_k)) = \bigcup_{i\in I} P_{1+k}(\overline{K_i}) = P_{1+k}(\overline{\bigcup_{i\in I}K_i})$ $= P_{1+k}(\bigcup_{i\in I}\overline{K_i})$, where the second equality is by normality of $P_{1+k}(K_i)$ with respect to $L(G_1\parallel G_k)$ and $Q_{k}$, for $i\in I$. Since the case for the projection $P_{2+k}$ is analogous, the proof is complete.
  \qed\end{proof}

  Let $\supccn(K, L, (\Sigma_{1,u}, \Sigma_{2,u}, \Sigma_{k,u}), (Q_{1+k},Q_{2+k}))$ denote the supremal conditionally controllable and conditionally normal sublanguage of $K$ with respect to the plant $L=L(G_1\parallel G_2\parallel G_k)$, the sets of uncontrollable events $\Sigma_{1,u}$, $\Sigma_{2,u}$, $\Sigma_{k,u}$, and projections $Q_{1+k}$ and $Q_{2+k}$, where $Q_{i+k}:\Sigma_{i+k}^* \to \Sigma^*_{i+k,o}$, for $i=1,2$. Similarly as in our original approaches~\cite{JDEDS,scl11}, we define the supremal controllable languages for the local plants combined with the coordinator. Unlike~\cite{JDEDS,scl11}, no supremal controllable language is needed to take care of the coordinator part $P_k(K)$ of the specification and, therefore, the supremal controllable sublanguage for the coordinator part is replaced by the coordinator language itself in formulas below for local plants combined with the coordinator.
  
  Consider the setting of Problem~\ref{problem:controlsynthesis} and define the languages
  \begin{equation}\label{eqCN}
    \begin{aligned}
      \supcn_{1+k} & = \supcn(P_{1+k}(K), L(G_1\parallel G_k), \Sigma_{1+k,u},Q_{1+k})\\
      \supcn_{2+k} & = \supcn(P_{2+k}(K), L(G_2\parallel G_k), \Sigma_{2+k,u},Q_{2+k})
    \end{aligned}
  \end{equation}
  where $\supcn(K,L,\Sigma_u,Q)$ denotes the supremal controllable and normal sublanguage of $K$ with respect to $L$, $\Sigma_u$, and $Q$. In the case of full observation, we use the following notation.
  \begin{equation}\label{eq0}
    \begin{aligned}
      \supC_{1+k} & = \supC(P_{1+k}(K), L(G_1 \parallel G_k), \Sigma_{1+k,u})\\
      \supC_{2+k} & = \supC(P_{2+k}(K), L(G_2 \parallel G_k), \Sigma_{2+k,u})
    \end{aligned}
  \end{equation}

  We now generalize our earlier sufficient conditions for the coordinated computation of the supremal conditionally controllable and conditionally normal sublanguage. Namely, the sufficient conditions are formulated in terms of controllability and normality of the composition of local supervisors projected to the coordinator alphabet. The main constructive result that uses the same form of conditions to guarantee both conditional controllability and conditional normality in a maximally permissive way is stated below.
  
  \begin{theorem}\label{thm2b}
    Consider the setting of Problem~\ref{problem:controlsynthesis} and the languages defined in~(\ref{eqCN}). Assume that $\supcn_{1+k}$ and $\supcn_{2+k}$ are synchronously nonconflicting. If $P_k(\supcn_{1+k})\cap P_k(\supcn_{2+k})$ is controllable and normal with respect to $L(G_k)$, $\Sigma_{k,u}$, and $Q_k$, then $\supcn_{1+k} \parallel \supcn_{2+k} = \supccn(K, L, (\Sigma_{1,u}, \Sigma_{2,u}, \Sigma_{k,u}), (Q_{1+k},Q_{2+k}))$, where $L=L(G_1\parallel G_2\parallel G_k)$.
  \end{theorem}
  \begin{proof}
    Let $M = \supcn_{1+k} \parallel \supcn_{2+k}$ and $\supccn = \supccn(K, L, (\Sigma_{1+k,u}, \Sigma_{2+k,u},\Sigma_{k,u}), (Q_{1+k},Q_{2+k},Q_{k}))$. To prove that $M\subseteq \supccn$, we show that $M \subseteq P_{1+k}(K)\parallel P_{2+k}(K) = K$ (where the equality is by conditional decomposability) is conditionally controllable with respect to $L$, and conditionally normal with respect to $L$ and $Q_{1+k},Q_{2+k},Q_{k}$. 
    
    However, $\supcn_{1+k} \parallel P_k(M) = \supcn_{1+k} \parallel P_k(\supcn_{1+k}) \parallel P_k(\supcn_{2+k}) = \supcn_{1+k} \parallel P_k(\supcn_{2+k}) = P_{1+k}(M)$ implies, together with Lemma~\ref{feng}, that $P_{1+k}(M)$ is controllable with respect to $[L(G_1\parallel G_k)]\parallel L(G_k)=L(G_1\parallel G_k)$, since synchronous nonconflictingness of $\supcn_{1+k}$ and $\supcn_{2+k}$ implies synchronous nonconflictingness of $\supcn_{1+k}$ and $P_k(M)$. Moreover,
    $Q_{1+k}^{-1}Q_{1+k}(P_{1+k}(M)) \cap L(G_1\parallel G_k) = 
     Q_{1+k}^{-1}Q_{1+k}(\supcn_{1+k} \parallel P_k(M)) \parallel L(G_1\parallel G_k) \subseteq 
     Q_{1+k}^{-1}Q_{1+k}(\supcn_{1+k}) \parallel Q_{1+k}^{-1}Q_{1+k}(P_k(M)) \parallel L(G_1\parallel G_k) = Q_{1+k}^{-1}Q_{1+k}(\supcn_{1+k}) \parallel L(G_1\parallel G_k) \parallel Q_{k}^{-1}Q_{k}(P_k(M)) \parallel L(G_k) = \supcn_{1+k} \parallel P_k(M) = P_{1+k}(M)$, where the last equality is by normality of $\supcn_{1+k}$ and the assumption on $P_k(M)$.
    The case for $P_{2+k}(M)$ is analogous, hence $M\subseteq \supccn$.

    To prove the opposite inclusion, we show that $P_{i+k}(\supccn)\subseteq \supcn_{i+k}$, for $i=1,2$. Indeed, $P_{1+k}(\supccn) \subseteq P_{1+k}(K)$ is controllable and normal with respect to $L(G_1\parallel G_k)$, $\Sigma_{1+k,u}$, and $Q_{1+k}$ by definition, thus $P_{1+k}(\supccn)\subseteq \supcn_{1+k}$. The case of $P_{2+k}(\supccn)$ is analogous, hence $\supccn \subseteq M$.
  \qed\end{proof}
 
  The following corollary is an immediate consequence of the previous result for systems with full observation.
  \begin{cor}\label{thmNEW}
    Consider the setting of Problem~\ref{problem:controlsynthesis} and languages defined in~(\ref{eq0}). Assume that $\supC_{1+k}$ and $\supC_{2+k}$ are synchronously nonconflicting. If $P_k(\supC_{1+k}) \cap P_k(\supC_{2+k})$ is controllable with respect to $L(G_k)$ and $\Sigma_{k,u}$, then $\supC_{1+k} \parallel \supC_{2+k} = \supCC(K, L, (\Sigma_{1,u}, \Sigma_{2,u}, \Sigma_{k,u}))$, where $L=L(G_1\parallel G_2\parallel G_k)$.
  \end{cor}
 
  It turns out that controllability required in Corollary~\ref{thmNEW} is a weaker condition that our earlier conditions of observer and OCC (resp. LCC) properties. Formally the following claim holds true.
 
  \begin{proposition}\label{prop8}
    Consider the setting of Problem~\ref{problem:controlsynthesis} and the languages defined in~(\ref{eq0}). Assume that $\supC_{1+k}$ and $\supC_{2+k}$ are synchronously nonconflicting, and let the projection $P^{i+k}_k:(\Sigma_i\cup\Sigma_k)^*\to\Sigma_k^*$ be an $(P^{i+k}_i)^{-1}L(G_i)$-observer and OCC (resp. LCC) for $(P^{i+k}_i)^{-1}L(G_i)$, for $i=1,2$. Then $P^{1+k}_k(\supC_{1+k})\cap P^{2+k}_k(\supC_{2+k})$ is controllable with respect to $L(G_k)$ and $\Sigma_{k,u}$.
  \end{proposition}
  \begin{proof}
    Since $\Sigma_{1+k}\cap\Sigma_{2+k}=\Sigma_k$, Lemma~\ref{lemma:Wonham} implies that $P_k(\supC_{1+k})\cap P_k(\supC_{2+k}) = P_k(\supC_{1+k}\parallel \supC_{2+k})$. By Lemma~\ref{obsComposition}, because $P_k^k = id$ is an $L(G_k)$-observer, $P_k$ is an $L:=L(G_1\parallel G_2\parallel G_k)$-observer. Assume that $t\in \overline{P_k(\supC_{1+k}\parallel \supC_{2+k})}$, $u\in\Sigma_{k,u}$, and $tu\in P_k(L)=L(G_k)$. Then there exists a string $s\in \overline{\supC_{1+k}\parallel \supC_{2+k}}\subseteq L$ such that $P_k(s)=t$. By the observer property, there exists $v$ such that $sv\in L$ and $P_k(sv)=tu$, that is, $v=v_1u$ with $P_k(v_1u)=u$. By the OCC property, $v_1\in\Sigma_u^*$, and by controllability of $\supC_{i+k}$, $i=1,2$, we get $sv_1u\in \overline{\supC_{1+k}} \parallel \overline{\supC_{2+k}} = \overline{\supC_{1+k}\parallel \supC_{2+k}}$, hence $tu\in \overline{P_k(\supC_{1+k}\parallel \supC_{2+k})}$.
  
    Similarly for LCC: from $sv=sv_1u\in L$, by the LCC property, there exists $v_2\in (\Sigma_u\setminus\Sigma_k)^*$ such that $sv_2u\in L$, and by controllability of $\supC_{i+k}$, $i=1,2$, $sv_2u\in \overline{\supC_{1+k}} \parallel \overline{\supC_{2+k}} = \overline{\supC_{1+k}\parallel \supC_{2+k}}$, hence the string $tu\in \overline{P_k(\supC_{1+k}\parallel \supC_{2+k})}$.
  \qed\end{proof}

\subsection{General Case and Complexity}  
  Finally, an important and interesting by-product of the relaxed framework is presented. We believed in the past that a solution to our coordination control problem (in terms of a conditionally controllable sublanguage) can only be computed as a product of languages in some special cases, where sufficient conditions such as those presented in~\cite{JDEDS} hold.  However, due to the above presented relaxation it becomes clear that such a solution can always be computed in our distributed way for prefix-closed specifications. Namely, it suffices to make the resulting language $P_k(\supcn_{1+k}) \cap P_k(\supcn_{2+k})$ controllable and normal with respect to $L(G_k)$ (as required in Theorem~\ref{thm2b}) by simply computing a supervisor for it as the following result suggests. Interestingly, we rediscover this way the role of the supervisor on the coordinator alphabet that is actually postponed to the very end of the coordination control synthesis and it is naturally used only when needed (that is, when controllability of $P_k(\supcn_{1+k}) \cap P_k(\supcn_{2+k})$ with respect to $L(G_k)$ does not hold).

  \begin{theorem}
    Consider the setting of Problem~\ref{problem:controlsynthesis} and languages defined in~(\ref{eqCN}). Let $\supCN'_k = \supCN(P_k(\supcn_{1+k}) \cap P_k(\supcn_{2+k}), L(G_k), \Sigma_{k,u},Q_{k})$. If the languages $\supcn_{i+k}$ and $\supcn_{k}'$ are synchronously nonconflicting (e.g., prefix-closed) for $i=1,2$, then $\supCN'_k \parallel \supCN_{1+k} \parallel \supCN_{2+k}$ is a conditionally controllable and conditionally normal sublanguage of $K$.
  \end{theorem}
  \begin{proof}
    Let $M=\supCN'_k \parallel \supCN_{1+k}\parallel \supCN_{2+k}$ and $\supCCN=\supccn(K, L, (\Sigma_{1+k,u}, \Sigma_{2+k,u}, \Sigma_{k,u}), (Q_{1+k},Q_{2+k})$, where $L=L(G_1\parallel G_2\parallel G_k)$. Then we have that $P_{i+k}(M)
    = \supCN_{i+k} \parallel P_k(\supCN_{j+k}) \parallel $ $\supCN'_k
    = \supCN_{i+k} \parallel P_k(\supCN_{1+k} \cap \supCN_{2+k}) \parallel \supCN'_k
    = \supCN_{i+k} \parallel \supCN'_k$, for $i=1,2$ and $j\neq i$.
    Hence, combining Lemmata~\ref{feng} and~\ref{fengpo}, we obtain that $P_{i+k}(M)$ is controllable with respect to $L(G_i \parallel G_k) \parallel L(G_k) = L(G_i \parallel G_k)$ and $\Sigma_{1+k,u}$, for $i=1,2$, and normal with respect to the same languages and $Q_{i+k}$. Therefore, $M\subseteq \supccn$.
  \qed\end{proof}

  An immediate consequence for systems with full observations follows.
  \begin{cor}\label{propM1}
    Consider the setting of Problem~\ref{problem:controlsynthesis} and languages defined in~(\ref{eq0}). Let $\supC'_k = \supC(P_k(\supC_{1+k}) \cap P_k(\supC_{2+k}), L(G_k), \Sigma_{k,u})$. Then $\supC'_k \parallel \supC_{1+k}\parallel  \supC_{2+k}$  is conditionally controllable sublanguage of $K$.
  \end{cor}

  It should be noted that the supervisor for the coordinator can be computed in a distributed way, that is, we can compute $\supCN'_k = \supCN(P_k(\supcn_{1+k}), L_k, \Sigma_{k,u}) \cap \supCN(P_k(\supcn_{2+k}), L_k, \Sigma_{k,u})$, or in the full observation case $\supC'_k = \cap_{i=1}^2 \supC( P_k(\supC_{i+k}), L(G_k), \Sigma_{k,u})$. This distributed computation is possible because the plant language is the same in both components of the intersection ($L(G_k)$) and $L(G_k)$ is trivially mutually controllable with respect to itself~\cite{KvSGM08}. Such a distributed computation is also important from the complexity viewpoint, because instead of computing the supervisor for the projection of composition of local supervisors, separate supervisors are computed for projection of individual local supervisors. Similarly as in modular control, their composition (here intersection) is never computed and they operate locally in conjunction with local supervisors $\supcn_{i+k}$, $i=1,2$.
  
  One could also find problematic our assumptions that the languages are required to be synchronously nonconflicting. For instance, these assumptions are trivially satisfied for prefix-closed languages, so they are not needed if one considers only prefix-closed languages. If general, non-prefix-closed languages are considered, it is known that to verify whether a synchronous product (of an unspecified number) of generators is synchronously nonconflicting is a PSPACE-complete problem~\cite{rohloff}. However, it is only the worst case and some optimization techniques could be found in the literature, see e.g.~\cite{FlordalM2006}, or a maximal nonconflicting sublanguage can be computed, cf.~\cite{Chen1991105}. Moreover, the good news of the PSPACE-completeness is that it is still computable in polynomial space. It should also be mentioned that if the number of components is fixed, the problem becomes tractable (namely, it is complete for the class NL of problems solvable in nondeterministic logarithmic space). For the general case with non-prefix-closed languages we have proposed in~\cite{JDEDS}a procedure based on abstraction for computing coordinators for nonblocking, which are needed if local supervisors $\supCN_{i+k}$ are conflicting.

  Finally, it is to be discussed whether it is advantageous to always postpone the computation of the supervisor on the coordinator alphabet. It appears that the former formula $\supC_k = \supC (P_k(K), L(G_k), \Sigma_{k,u})$ that has been integrated within the coordination control synthesis with full observation is computationally simpler then the postponed supervisor $\supC'_k$ above. However, in general, we cannot guarantee that $\supC_{1+k}\parallel \supC_{2+k}$ is conditionally controllable, which is needed for being able to synthesize this language within our coordination control architecture. This suggests that the best strategy is to postpone the computation of a supervisor on the coordinator alphabet and compute $\supC'_k$ instead of $\supC_k$ that we have advocated in our original approach. In the opposite case, where some of the sufficient conditions for conditional controllability of $\supC_{1+k}\parallel \supC_{2+k}$ is satisfied (which may be known in advance especially for the observer and LCC conditions), it is better to use the former approach for the computational reasons. Similarly, in the case $\supC_k \subseteq P_{k}(\supC_{i+k})$, for $i=1,2$, the supremal conditionally controllable sublanguages can be computed in a distributed way even for non-prefix-closed specifications as we have shown in~\cite[Theorem 6]{JDEDS}, hence the a posteriori supervisor is not needed.

\subsection{An example}
  Let $G=G_1\parallel G_2$ be a system over the event set $\Sigma=\{a_1,a_2,c,u,u_1,u_2\}$, where $G_1$ and $G_2$ are defined in Fig.~\ref{figA}, $\Sigma_u=\{u,u_1,u_2\}$.
  \begin{figure}
   \centering
   \includegraphics[scale=.8]{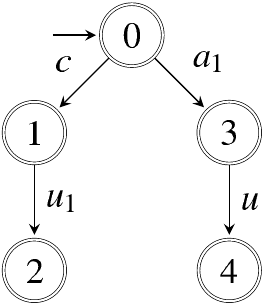}
   \qquad
   \includegraphics[scale=.8]{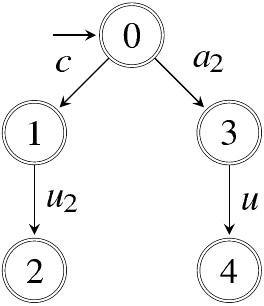}
   \qquad
   \includegraphics[scale=.8]{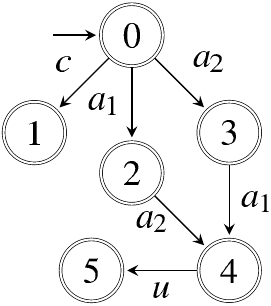}
   \qquad
   \includegraphics[scale=.8]{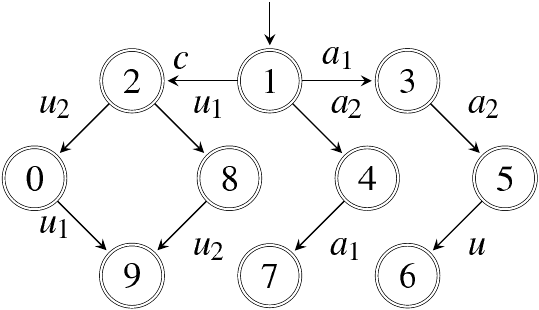}
   \caption{Generators $G_1$, $G_2$, the coordinator, and the generator for the specification $K$.}
   \label{figA}
   \label{figK}
  \end{figure}
  The specification $K$ is defined in Fig.~\ref{figK} (right).
  The coordinator event set $\Sigma_k$ has to contain shared events $c$ and $u$, and to make $K$ conditionally decomposable, at least one of $a_1, a_2$ has to be in $\Sigma_k$. The coordinator event set $\Sigma_k$ has to contain shared events $c$ and $u$, and to make $K$ conditionally decomposable, at least one of $a_1, a_2$ has to be in $\Sigma_k$. Let $\Sigma_k=\{a_2,c,u\}$.
  
  If we required the projections to satisfy the observer and OCC properties, we need to add also $a_1$ to $\Sigma_k$. The coordinator is then defined as $G_k = P_{k}(G_1)\parallel P_{k}(G_2)$, see Fig.~\ref{figA},
   $P_{1+k}(K) = \overline{\{a_1a_2u,a_2a_1,cu_1\}}$, and
   $P_{2+k}(K) = \overline{\{a_1a_2u,a_2a_1,cu_2\}}$.
   This gives that $\supCN_{1+k} = \overline{\{a_1a_2u,a_2,cu_1\}}$ and
   $\supCN_{2+k} = \overline{\{a_1a_2u,a_2,cu_2\}}$,
  whose synchronous product results in the supremal conditionally-controllable sublanguage
  $\overline{\{a_1a_2u,a_2,cu_1u_2,cu_2u_1\}}$ of $K$ that coincides with the supremal controllable sublanguage of $K$.
 
  However, note that $P_k(\supcn_{1+k})\cap P_k(\supcn_{2+k})$ is controllable and normal with respect to $L(G_k)$, hence, by Theorem~\ref{thm2b}, $\supcn_{1+k} \parallel \supcn_{2+k} = \supccn(K, L, (\Sigma_{1,u}, \Sigma_{2,u}, \Sigma_{k,u}), (Q_{1+k},Q_{2+k}))=\overline{\{a_1a_2u,a_2,cu_1u_2,cu_2u_1\}}$. In more detail, the coordinator is now defined as $G_k = P_{k}(G_1)\parallel P_{k}(G_2)$, 
  \begin{figure}[h]
    \centering
    \includegraphics[scale=.5]{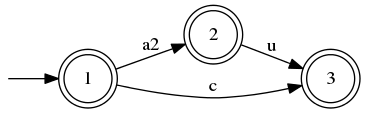}
    \caption{Coordinator $G_k$ for the second case.}
    \label{figB}
  \end{figure}
  see Fig.~\ref{figB}. Then,
   $P_{1+k}(K) = \overline{\{a_1a_2u,a_2a_1,cu_1\}}$, and
   $P_{2+k}(K) = \overline{\{a_2u,cu_2\}}$.
   $\supCN_{1+k} = \overline{\{a_1a_2u,a_2,cu_1\}}$,
   $\supCN_{2+k} = \overline{\{a_2u,cu_2\}}$,
  whose synchronous product results in the supremal conditionally-controllable sublanguage
  $\overline{\{a_1a_2u,a_2,cu_1u_2,cu_2u_1\}}$ of $K$ that coincides with the supremal controllable sublanguage of $K$.

\section{Conclusion} \label{sec:conclusion}
  We have proposed a relaxed definition of the coordination control problem, where the role of a supervisor for the coordinator is postponed until the final phase of the coordination control synthesis. This has led to the relaxed forms of conditional controllability, as well as conditional observability and conditional normality for partially observed discrete-event systems. We have reformulated the main existential and constructive results of coordination control for the relaxed problem. In particular, a weaker (than existing ones) sufficient condition for a distributed computation of supremal conditionally controllable languages have been found in Corollary~\ref{thmNEW}. The corresponding sufficient condition for the distributed computation of supremal conditionally controllable and conditionally normal languages is presented in Theorem~\ref{thm2b}.

  Moreover, we have shown that in the relaxed framework a conditionally controllable and conditionally normal sublanguage can always be computed using a distributive computational scheme that we have proposed in our earlier papers. Unlike our previous approaches, no restrictions are now needed, because the weaker sufficient condition can always be guaranteed by postponing the role of the supervisor on the coordinator alphabet.
 
  We emphasize that our approach can easily be extended to the general case of $n$ subsystems running in parallel. The (single) central coordinator will have in its alphabet all events shared by at least two subsystems. All concepts and results can be extended from the generic case $n=2$ in a straightforward manner. However, with an increasing number of components it is likely that the coordinator will, in some situations, grow too much, e.g., too many events will have to be included into the coordinator alphabet to make the global specification conditionally decomposable. We have therefore recently proposed a multi-level coordination control with a hierarchical structure of subsystems into groups and their coordinators. Another reason why to advocate the use of decentralized coordination with a multi-level structure of subsystems and coordinators is that, in practice, there are often multiple specifications naturally decomposed into smaller specifications, while each specification is over an alphabet of a group of subsystems.

\appendix
\section{Auxiliary results}\label{appendix}
  In this section, we list auxiliary results required in the paper.
  
  \begin{lemma}[Proposition~4.6, \cite{FLT}]\label{feng}
    For $i=1,2$, let $L_i$ be a prefix-closed language over the alphabet $\Sigma_i$, and let $K_i$ be a controllable sublanguage of $L_i$ with respect to $L_i$ and $\Sigma_{i,u}$. Let $\Sigma=\Sigma_1\cup \Sigma_2$. If $K_1$ and $K_2$ are synchronously nonconflicting, then $K_1\parallel K_2$ is controllable with respect to $L_1\parallel L_2$ and $\Sigma_u$.
    \qed
  \end{lemma}

  \begin{lemma}\label{obsComposition2}
    For $i=1,2$, let $L_i$ be a prefix-closed language over the alphabet $\Sigma_i$, and let $K_i$ be an observable sublanguage of $L_i$ with respect to $L_i$, $\Sigma_{i,u}$ and $Q_i:\Sigma_i^*\to \Sigma_{i,o}^*$. Let $\Sigma=\Sigma_1\cup \Sigma_2$. If $K_1$ and $K_2$ are synchronously nonconflicting, then $K_1\parallel K_2$ is observable with respect to $L_1\parallel L_2$, $\Sigma_u$ and $Q:\Sigma^*\to \Sigma_{o}^*$.
  \end{lemma}
  \begin{proof}
    Let $s, s'\in \Sigma^*$ be such that $Q(s) = Q(s')$. Let $\sigma\in \Sigma$ and assume that $s\sigma \in \overline{K_1\parallel K_2}$, $s' \in \overline{K_1\parallel K_2}$, and $s'\sigma \in L_1\parallel L_2$. Let $P_i:\Sigma^*\to \Sigma_i^*$, for $i=1,2$. Then $P_i(s\sigma)\in \overline{K_i}$, $P_i(s')\in \overline{K_i}$, and $P_i(s'\sigma)\in L_i$ imply that $P_i(s'\sigma)\in \overline{K_i}$, by observability of $K_i$ with respect to $L_i$. Thus, $s'\sigma\in \overline{K_1}\parallel \overline{K_2} = \overline{K_1 \parallel K_2}$.
  \end{proof}

  \begin{lemma}\label{fengpo}
    Let $K_1\subseteq L_1$ over $\Sigma_1$ and $K_2\subseteq L_2$ over $\Sigma_2$ be languages such that $K_1$ is normal with respect to $L_1$ and $Q_1:\Sigma_1^*\to \Sigma_{1,o}^*$ and $K_2$ is normal with respect to $L_2$ and $Q_2:\Sigma_2^*\to \Sigma_{2,o}^*$. If $K_1$ and $K_2$ are synchronously nonconflicting, then $K_1\parallel K_2$ is normal with respect to $L_1\parallel L_2$ and $Q:(\Sigma_1\cup\Sigma_2)^*\to (\Sigma_{1,o}\cup\Sigma_{2,o})^*$.
  \end{lemma}
  \begin{proof}
    $Q^{-1}Q(\overline{K_1\parallel K_2}) \cap L_1\parallel L_2 \subseteq Q_1^{-1}Q_1(\overline{K_1}) \parallel Q_2^{-1}Q_2(\overline{K_2}) \parallel L_1\parallel L_2 = \overline{K_1} \parallel \overline{K_2} = \overline{K_1\parallel K_2}$. As the other inclusion always holds, the proof is complete.
  \end{proof}
  
  \begin{lemma}[\cite{Won04}]\label{lemma:Wonham}
    Let $P_k : \Sigma^*\to \Sigma_k^*$ be a projection, and let $L_i \subseteq \Sigma_i^*$, where $\Sigma_i\subseteq \Sigma$, for $i=1,2$, and $\Sigma_1\cap \Sigma_2 \subseteq \Sigma_k$. Then $P_k(L_1\parallel L_2)=P_k(L_1) \parallel P_k(L_2)$. 
  \qed
  \end{lemma}

  \begin{lemma}[\cite{pcl06}]\label{obsComposition}
    For $i\in J$, let $L_i \subseteq \Sigma_i^*$ be a language, and let $\cup_{k,\ell\in J}^{k\neq\ell} (\Sigma_k\cap \Sigma_\ell)\subseteq \Sigma_0 \subseteq (\cup_{i\in J} \Sigma_i)^*$. If $P_{i,0}:\Sigma_i^* \to (\Sigma_i\cap \Sigma_0)^*$ is an $L_i$-observer, for $i\in J$, then $P_{0}:(\cup_{i\in J} \Sigma_i)^* \to \Sigma_0^*$ is an $(\parallel_{i\in J} L_i)$-observer. 
    \qed
  \end{lemma}

\section*{Acknowledgements}
  The research was supported by RVO 67985840, by the Czech Ministry of Education in project MUSIC (grant LH13012), and by the DFG in project DIAMOND (Emmy Noether grant KR~4381/1-1).
 
\bibliographystyle{model1-num-names}
\bibliography{revision2}

\end{document}